\documentclass[11pt]{amsart}
\usepackage{amsmath, amsthm, amsfonts, amssymb}
\usepackage{mathrsfs}
\usepackage{cite}
\usepackage{graphicx, xcolor}
\usepackage[colorlinks]{hyperref}

\newtheorem{theorem}{Theorem}[section]
\newtheorem{lemma}[theorem]{Lemma}

\newtheorem{corollary}[theorem]{Corollary}
\newtheorem{remark}[theorem]{Remark}

\DeclareMathOperator{\Aut}{Aut}

\DeclareMathOperator{\SL}{SL}
\DeclareMathOperator{\Fit}{Fit}
\DeclareMathOperator{\FATR}{FATR}

\newcommand{\sub}{\leqslant}
\newcommand{\nsub}{\nleqslant}

\newcommand{\C}{\mathcal{ C}}
\newcommand{\nil}{\mathfrak{Nil}}
\newcommand{\cyc}[1]{\langle #1\rangle}

\newcommand{\Core}{{\rm Core}}
\newcommand{\nor}{\trianglelefteq}
\newcommand{\nno}{\ntrianglelefteq}

\begin{document}
\title[\scriptsize{A note on meta and para-$\nil$-Hamiltonian groups}]{A note on meta and para-$\nil$-Hamiltonian groups}

\author[ Hamid Mousavi]{Hamid Mousavi}

\address{{Department of Pure Mathematics, Faculty of Mathematical Sciences, University of Tabriz, Tabriz, Iran.}}
\email{hmousavi@tabrizu.ac.ir}


\begin{abstract}
Let $\mathfrak{Nil}$ be the class of nilpotent groups. This article explores the finiteness of meta and para-$\mathfrak{Nil}$-Hamiltonian groups or their derived subgroups when these groups contain a  soluble subgroup of finite index or a non-nilpotent (or insoluble) subgroup of finite order respectively.
\end{abstract}
\maketitle
\section{Introduction}
Let $\nil$ be a class of nilpotent groups. {\it The group $G$ is said to be meta-$\nil$-Hamiltonian if all its non-nilpotent subgroups is normal}. Also, {\it we say that $G$ is para-$\nil$-Hamiltonian if $G$ is a non-nilpotent group and every non-normal subgroup of $G$ is either nilpotent or minimal non-nilpotent}. Also, {\it $G$ is called biminimal non-$\nil$ group if it is neither a nilpotent nor a minimal non-nilpotent group, but each proper subgroup of $G$ either is nilpotent or is a minimal non-nilpotent group}. Para-$\nil$-Hamiltonian groups are a natural extension of  biminimal non-nilpotent groups.

If $\mathfrak{A}$ is the class of abelian groups, then the class of meta-$\mathfrak{A}$-Hamiltonian is called metahamiltonian and the class of para-$\mathfrak{A}$-Hamiltonian groups is called parahamiltonian. Many researchers have been done on metahamiltonian groups, which can be referred to \cite{BFT, FGM, FGM1, FangAn, RoSea1, RoSea2, RoSea3}.

Recall that a group $G$ is known as locally graded if every non-trivial finitely generated subgroup of $G$ contains a proper subgroup of finite index.

In the previous article~\cite{DM}, it is shown that the para-$\nil$-Hamiltonian group $G$, if contains a minimal non-nilpotent subgroup whose normal closure is locally graded, then $G'$ is finite and if $G$ is  insoluble, so $G$ is finite~\cite[Theorem 4.6, Corollary 4.7]{DM}. Also any finite insoluble para-$\nil$-Hamiltonian group is isomorphic either to $A_5$ or $\SL(2,5)$ (\cite[Theorem 3.6]{DM}).
 
Also, the structure of insoluble meta-$\nil$-Hamiltonian groups has been investigated in two perfect and non-perfect cases (\cite[Propositions 5.8, 5.9, 5.10]{DM}). Additionally, it has been shown that every locally graded meta-$\nil$-Hamiltonian group is solvable~\cite[Lemma 5.11]{DM}.

In this article (Theorem~\ref{T-W}), we will demonstrate that if $G$ is a soluble-by-finite meta-$\mathfrak{Nil}$-Hamiltonian group, then every finitely generated non-nilpotent subgroup of $G$ possesses a non-nilpotent homomorphic image. These groups are called $\mathcal{W}$-groups. De Mari and de Giovanni~\cite[Theorem A]{MaGi} showed that if a $\mathcal{W}$-group $G$ has finitely many normalizers of a non-locally nilpotent subgroup, then either $G$ is locally nilpotent or its commutator subgroup $G'$ is finite.
Since a non-nilpotent finitely generated meta-$\mathfrak{Nil}$-Hamiltonian group is not locally nilpotent, then $G'$ must be finite, as finite meta-$\mathfrak{Nil}$-Hamiltonian group is soluble~\cite[Theorem 3.10]{DM}, so $G'$ is soluble, and consequently, $G$ is also soluble (Corollary~\ref{C-3.6}). We have an updated version of Corollary~\ref{C-3.6} by replacing the assumption of soluble-by-finite with (torsion-free soluble)-by-finite. This demonstrates that $G$ is polycyclic and has a nilpotent Fitting subgroup of finite index, as stated in Theorem~\ref{T-3.8}.

Also, In this article, it will be shown that, if the para-$\nil$-Hamiltonian group contains a finite normal insoluble subgroup, then $G$ is finite (Theorem~\ref{T-3.10}). Additionally, an insoluble para-$\nil$-Hamiltonian group is finite if it contains a nilpotent subgroup of finite index or a finite non-nilpotent subgroup with a finite normal closure (Theorems ~\ref{T-3.11}, \ref{T-3.12}).

\section{Primary and definition}
A group is called minimax, if it has a series of finite length whose factors either satisfy the minimal or the maximal condition. 

Now we need the following theorem of  Amberg, Franciosi, and de Giovanni (1988).
\begin{theorem}\cite[Theorem 6.3.4]{AmFrGi}\label{Amb}
 Let the group $G = AB = AK = BK$ be the product of three nilpotent subgroups $A$, $B$, and $K$, where $K$ is normal in $G$. If $K$ is minimax, then $G$ is nilpotent. 
\end{theorem}

\begin{lemma}\cite[4.6.4]{LeRob}\label{L-2.2}
Let $G$ be a finitely generated soluble group. Then $G$ is polycyclic if and only if $x^{\cyc{g}}$ is finitely generated for all $x, g \in G$.
\end{lemma}

If $G$ is nilpotenen, then $G'\sub\Phi(G)$. The converse is true for ploycyclic group.
\begin{lemma}\cite[4.5.20]{Rob}\label{L-2.2.1}
If G is a polycyclic group and $G'\sub \Phi(G)$, then $G$ is nilpotent. 
\end{lemma}

A soluble group has finite abelian total rank, or is a soluble FATR-group, if it has a series in which each factor is abelian of finite total rank (see~\cite[Page: 85 ]{LeRob}).

\begin{lemma}\cite[5.1.6 (ii)]{LeRob}\label{L-2.3}
Let G be a soluble group. Then $G$ has $\FATR$  if and only if $G$ is poly-(cyclic, quasicyclic, or rational).
\end{lemma}

As a result of Gruenberg (1961) and Mal’cev (1951), the Fitting subgroup of  a soluble group with finite abelian total rank, is nilpotent.
\begin{theorem}\cite[5.2.2]{LeRob}\label{T-2.4}
 Let $G$ be a soluble group with finite abelian total rank. Then $\Fit(G)$ is nilpotent and $G/ \Fit(G)$ is abelian-by-finite. Thus $G$
is nilpotent-by-abelian-by-finite.
\end{theorem}

By a result of Robinson (1968), any maximal subgroup of soluble-by-finite group with finite abelian ranks, is of finite index (\cite[5.2.12]{LeRob}).

Recall that a group $G$ is called a $\mathcal{W}$-group if every finitely generated non-nilpotent subgroup of $G$ has a finite, non-nilpotent, homomorphic image. 

\begin{theorem}\cite[Theorem A]{MaGi}\label{T-2.6}
Let $G$ be a $\mathcal{W}$-group with finitely many normalizers of non- (locally nilpotent) subgroups. Then either $G$ is locally nilpotent or its commutator subgroup $G'$ is finite.
\end{theorem}
\begin{lemma}\cite[Corollary 2.3.]{MaGi}\label{L-2.7}
 Let $G$ be a locally graded group with finitely many normalizers of  non-nilpotent subgroups. Then $G$ is soluble-by-finite.
\end{lemma}
\section{ Meta-$\mathfrak{Nil}$-Hamiltonian group}
Let $G$ be non-nilpotent meta-$\nil$-Hamiltonian group and $H$ be a non-normal subgroup of $G$. Then $H$ must be nilpotent. Therefore $G$ is a para-$\nil$-Hamiltonian group such that all of whose non-normal subgroups are nilpotent. So we must focus on normal subgroup of $G$, in particular normal subgroup of finite index.

\begin{lemma}\label{L-3.4}
Let $G$ be a finitely generated soluble group with torsion-free abelian subgroup of finite index. Then $G$ is polycyclic and $\Fit(G)$ is nilpotent of finite index.
\end{lemma}
\begin{proof}
Suppose that $A$ is a torsion-free abelian subgroup of finite index in $G$. Hence $|G: A_G|$ is also finite. In particular, $N=A_G$ is finitely generated, torsion-free abelian group.

Assume that $x, g\in G$ and $L=x^{\cyc{g}}$. As $ L/(L\cap N)\cong LN/N\sub G/N$ is finite and $L\cap N$ is finitely generated, for $N$ is free abelian group, so $L$ is finitely generated. Now, by Lemma~\ref{L-2.2}, $G$ is polycyclic and by Lemma~\ref{L-2.3}, $G$ is $\FATR$-group. Therefore, $\Fit(G)$ is nilpotent by Theorem~\ref{T-2.4}. As $A\sub \Fit(G)$, so $\Fit(G)$ is of finite index.
\end{proof}

\begin{theorem}\label{T-W}
Let  $G$ be a meta-$\mathfrak{Nil}$-Hamiltonian group. If $G$ is a soluble-by-finite, then $G$ is  a $\mathcal{W}$-group. 
\end{theorem}
\begin{proof}
Let $S$ be a soluble group of finite index in $G$. As finite meta-$\mathfrak{Nil}$-Hamiltonian group is soluble and $G/S_G$ is finite, so $G$ is soluble. Assume that $H$ is a finitely generates non-nilpotent subgroup of $G$. By a result of Robinson,~\cite[4.4.4]{LeRob}, $H$ has a finite non-nilpotent homomorphic image. Therefore $G$ is a $\mathcal{W}$-group.
\end{proof}

\begin{corollary}\label{C-3.6}
Let  $G$ be a finitely generated non-nilpotent meta-$\mathfrak{Nil}$-Hamiltonian group. If $G$ is a soluble-by-finite, then $G'$ is finite.
\end{corollary}
\begin{proof}
By Theorem~\ref{T-W}, $G$ is a $\mathcal{W}$-group. As $G$ is not locally-nilpotent, by Theorem~\ref{T-2.6}, $G'$ is finite.
\end{proof}

\begin{corollary}\label{C-3.7}
Let  $G$ be a locally graded  meta-$\mathfrak{Nil}$-Hamiltonian group. Then either $G$ is locally nilpotent or $G'$ is finite.
\end{corollary}
\begin{proof}
By Lemma~\ref{L-2.7}, $G$ is soluble-by-finite. So $G$ is a $\mathcal{W}$-group (by Theorem~\ref{T-W}). Now by Theorem~\ref{T-2.6}, the result is obtained.
\end{proof}

\begin{theorem}\label{T-3.8}
Let $G$ be a finitely generated non-nilpotent meta-$\mathfrak{Nil}$-Hamiltonian group. If $G$ has a torsion-free soluble subgroup of finite index, then:
\begin{itemize}
\item[(i)] $G'$ is finite and $G$ is soluble;
\item[(ii)] $G$ is center-by-finite;
\item[(iii)]$G$ is polycyclic and $\Fit(G)$ is nilpotent of finite index;
\item[(iv)] there are only finitely many maximal subgroups $M$ such that $G=MG'$, obviously $M$ is nilpotent.
\end{itemize}
\end{theorem}
\begin{proof}
 By Corollary~\ref{C-3.6},  $G'$ is finite and $G$ is soluble. Assume that $S$ is torsion-free subgroup of finite index. As  $|G: S_G|$ is finite, $S_G$ is finitely generated, torsion-free soluble subgroup of $G$. Hence $S_G\cap G'=1$ and so $S_G\sub Z(G)$ is torsion-free abelian of finite rank, where $Z(G)$ is the center of $G$. Therefore $|G: Z(G)|$ is finite.
Now by Lemma~\ref{L-3.4}, $G$ is a ploycyclic group with nilpotent Fitting subgroup of finite index. 

Since $G$ is a non-nilpotent polycyclic group, it has a non-normal maximal subgroup $M$ (by Lemma~\ref{L-2.2.1}) that is nilpotent. Thus, we have $G' \nsub M$, which follows that $G = MG'$. The number of these maximal subgroups $M$ is finite since $G$ is finitely generated and $G'$ is finite (by \cite[2.7]{DKS}).
\end{proof}

Theorem~\ref{T-3.8} is also an updated version of~\cite[Theorem 3.5]{DM}, where the assumption of (torsion-free nilpotent)-by-finite is replaced with (torsion-free soluble)-by-finite. Unfortunately, the proof in~\cite[Proposition 5.5]{DM} contains an error.

\begin{theorem}\label{T-3.9}
Let $G$ be a  non-nilpotent meta-$\mathfrak{Nil}$-Hamiltonian group. If $G$ is polycyclic-by-finite, then $G'$ is finite, $G$ is finite-by-polycyclic and $\Fit(G)$ is nilpotent-by-finite.
\end{theorem}
\begin{proof}
Without losing the generality, we can assume that $G$ is infinite. By assumption, $G$ contains a polycyclic normal subgroup $H$ of finite index. Since $H$ is finitely generated, thus $G$ is finitely generated non-nilpotent soluble group. By Corollary~\ref{C-3.6}, $G'$ is finite. Assume that $T_H$ is the torsion subgroup of $H$, by~\cite[2.49(ii)]{DKS}, $T_H$ is finite. As $H/T_H\neq 1$ is a torsion free subgroup of $G/T_H$ and $(G/T_H)'$ is finite, so $H/T_H$ is a central subgroup of finite index. If $G/T_H$ is nilpotent, there is nothing more to prove. Assume that $G/T_H$ is non-nilpotent. By Lemma~\ref{L-3.4}, $G/T_H$ is polycyclic and $\Fit(G/T_H)$ is nilpotent. As $\Fit(G)T_H/T_H\sub\Fit(G/T_H)$, so $\Fit(G)/(\Fit(G)\cap T_H)$ is nilpotent. Therefore $\Fit(G)$ is nilpotent-by-finite.
\end{proof}

Any locally graded meta-$\nil$-Hamiltonian group is soluble by~\cite[Lemma 5.11]{DM}, so every finite meta-$\nil$-Hamiltonian group is soluble. In the following remark, we examine the insoluble meta-$\nil$-Hamiltonian groups.

\begin{remark} \label{K-3.13}
Let $G$ be an insoluble meta-$\nil$-Hamiltonian group. Then $G$ is infinite.

\item[(a)] Assume that $G$ is perfect, then $G$ is minimal non-nilpotent.
\begin{itemize}
\item[(a-i)] If $G$ has a maximal subgroup, then $G$ is finitely generated and $G/\Phi(G)$ is non-abelian simple group by~\cite[propositions 5.8, 5,9]{DM}. In additional, by Corollary~\ref{C-3.6}, $G$ does not contain any soluble subgroup of finite index, so any maximal subgroup of $G$ is non-normal of infinite index and every finite normal subgroup of $G$ is contained in $\Phi(G)$ and is central by $N/C$-Theorem.

\item[(a-ii)] Otherwise $G$ is Fitting $p$-group for some prime $p$ by~\cite[Theorem 3.3 (i) and (ii)]{Smith}. 
\end{itemize}

\item[(b)] Assume that $G$ is not perfect, then $G''=\gamma_3(G)$ is perfect insoluble and is satisfied (a).
\end{remark}

\begin{theorem}
Let $G$ be a finitely generated insoluble meta-$\nil$-Hamiltonian. If $G$ is periodic, then:
\begin{itemize}
\item[(i)] $G''=\gamma_3(G)$ is perfect minimal non-nilpotent of finite index;
\item[(ii)]  the Frattini factor group of  $\gamma_3(G)$ is non-abelian simple group;
\item[(iii)] $\gamma_3(G)$ is intersection of all subgroups of finite index;
\item[(iv)] $\gamma_3(G)$ does not any subgroup of finite index.
\end{itemize} 
\end{theorem}
\begin{proof}
(i) If $G$ is perfect, then $G=G''=\gamma_3(G)$ is minimal non-nilpotent group and if $G'\neq G$, then $G''=\gamma_3(G)$ is perfect insoluble group and so is minimal non-nilpotent (Remark~\ref{K-3.13}-(b)). As $G/G''$ is finitely generated nilpotent periodic group, so if finite. 

(ii) By (i), $\gamma_3(G)$ is a finitely generated insoluble group. Since every finitely generated group have a maximal subgroup, so  $\gamma_3(G)/\Phi(\gamma_3(G))$ is non-abelian simple group (Remark~\ref{K-3.13}-(a)).

(iii) Suppose that $H\sub G$ is a subgroup of  finite index. If $\gamma_3(G)\nsub H$, then $\gamma_3(G)\cap H$ is a nilpotent subgroup of finite index in $G$. By Corollary~\ref{C-3.6}, $G'$ is finite, so $G$ is soluble, which  leads to a contradiction. 

(iv) By (i), $\gamma_3(G)$ is of finite index, therefore, the result is obtained from (iii).
\end{proof}

\section{Para-$\mathfrak{Nil}$-Hamiltonian groups}
We start this section with the following lemma, which presents a descending series for finitely generated groups with a non-nilpotent subgroup of finite index. This descending series will play an important role in the proof of Theorem~\ref{T-3.11}.

\begin{lemma}\label{L-3.1}
Let $G$ be a finitely generated group  and $N$ be a torsion-free nilpotent normal subgroup of $G$ of finite index. Then for any tow distance primes $q>p>|G/N|$, there exist tow natural numbers $r$  and $s$ which are not zero together such that,  for any $i\geq r$ and $j\geq s$, the subgroup $N/N^{p^iq^j}$ has a complement $K_{i,j}/N^{p^iq^j}$ in $G/N^{p^iq^j}$. In particular, for any $i'\geq i$ and any $j'\geq j$ the subgroup $K_{i',j'}N^{p^iq^j}/N^{p^iq^j}$ is a complement of $N/N^{p^iq^j}$ and we can assume that $K_{i',j'}\sub K_{i,j}$. Hence for every $j$ we get the following sequence (of course, for every $1\leq t<i$, it is possible that $G = K_{t,j}$):
\[\cdots\sub K_{i,j}\sub\cdots\sub K_{r,j}\sub\cdots\sub K_{1,j}.\]
\end{lemma}
\begin{proof}
 Assume that $q>p>|G/N|$ are distance primes. We not that, as $G$ is finitely generated and $N$ is of finite index, so $N$ is finitely generated torsion-free and according to~\cite[5.2.21]{Rob}, $N$ is residually finite. Thus, there exist tow natural numbers $r$ and $s$, such that for any $i\geq r$ and $j\geq s$, the subgroups $N^{p^i}$ and $N^{q^j}$ are proper subgroups of  $N$.  Since
 \[|G:N^{p^iq^j}|=|G:N||N:N^{p^iq^j}|\] 
 and $N/N^{p^iq^j}$ is finite, since is finitely generated periodic nilpotent group, so $G/N^{p^iq^j}$ is finite. We note that  
 \[|G/N^{p^iq^j}: N/N^{p^iq^j}|=|G/N|<p,\] 
so 
$(|G/N^{p^iq^j}: N/N^{p^iq^j}|, |N/N^{p^iq^j}|)=1$. Hence by Schur-Zassenhaus Theorem, for some supplement $K_{i,j}$ of $N$,
 \[\frac{G}{N^{p^iq^j}}\cong \frac{N}{N^{p^iq^j}}\rtimes \frac{K_{i,j}}{N^{p^iq^j}}.\]
Therefore $G=NK_{i,j}$ and $N\cap K_{i,j}=N^{p^iq^j}$.

Now assume that $i'\geq i$, $j'\geq j$ and $T/N^{p^{i'}q^{j'}}$ is a complement of $N/N^{p^{i'}q^{j'}}$. As $G=NT$, $TN^{p^iq^j}/N^{p^iq^j}$ is a complement of $N/N^{p^iq^j}$, for
\[TN^{p^iq^j}\cap N=N^{p^iq^j}(T\cap N)=N^{p^iq^j}N^{p^{i'}q^{j'}}=N^{p^iq^j},\]
hence for some $g\in G$, $K_{i,j}=T^gN^{p^iq^j}$. By taking $K_{i',j'}=T^g$, $K_{i',j'}N^{p^iq^j}/N^{p^iq^j}$ is a complement of $N/N^{p^iq^j}$ and we get $K_{i',j'}\sub K_{i,j}$.
\end{proof}
Assume that $G$ is a finitely generated group and $N$ is a torsion-free nilpotent normal subgroup of finite index. 
By Lemma~\ref{L-3.1}, for any prime $p>|G/N|$, and for some $i\geq 1$ depending to $p$, we have \[\frac{G}{N^{p^j}}\cong \frac{N}{N^{p^j}}\rtimes \frac{K_j}{N^{p^j}},\]
where $j\geq i$. Also we can assume that  \[\cdots\sub K_{j+1}\sub K_j\sub\cdots\sub K_1.\] 
The series mentioned above is referred to as the dependent series for $p$ (of course, for every $1\leq t<i$, it is possible that $ K_t=G$).

\begin{lemma}\label{L-3.2}
Let $G$ be a finitely generated group and $N$ be a torsion-free nilpotent normal subgroup of finite index. If $G$ is non-nilpotent, then for any prime $p$, the dependent subgroups $K_i$ are non-nilpotent.
\end{lemma}
\begin{proof}
Assuming the verdict is false, let's suppose that $q>p>|G/N|$ are prime numbers, and there exist natural numbers $i$ and $j$ such that the dependent subgroups $K_i$ and $K_j$ respect to $p$ and $q$,  are nilpotent. Assume that $K_{i,j}/N^{p^iq^j}$ is a complement of $N/N^{p^iq^j}$ in $G/N^{p^iq^j}$. By Lemma~\ref{L-3.1}, $K_{i,j}N^{p^i}$ is conjugate with $K_i$ and $K_{i,j}N^{q^j}$ with $K_j$, so both are nilpotent. Since $|G:K_{i,j}N^{p^i}|=|N: N^{p^i}|$ is a finite power of $p$ and $|G:K_{i,j}N^{q^j}|=|N: N^{q^j}|$ is a finite power of $q$ and both are coprime to $|G/N|<p$, so by taking $A=K_{i,j}N^{p^i}$,  $B=K_{i,j}N^{q^j}$ and $C=N$, $G=AB=AC=BC$ is triple product of three nilpotent subgroups of coprime index. As $C$ is normal minimax subgroup of $G$, by Theorem~\ref{Amb}, $G$ is nilpotent which is a contradiction.
\end{proof}

\begin{corollary}\label{C-3.3}
Let $G$ be a finitely generated group and $H$ be a nilpotent subgroup of finite index. Then, either $G$ finite-by-nilpotent or $G$ is not minimal non-nilpotent.
\end{corollary}
\begin{proof}
Assume that $G$ is not finite-by-nilpotent. Since $H_G=\Core_G(H)$ is a finitely generated nilpotent group, so $G$ is polycyclic-by-finite and $T_H$ the torsion subgroup of $H_G$ is finite (by~\cite[2.49(ii)]{DKS}). As $H_G/T_H\neq 1$ is a torsion-free nilpotent subgroup of finite index of a finitely generated non-nilpotent group $G/T_H$, so by Lemma~\ref{L-3.2}, $G/T_H$ contains a proper non-nilpotent subgroup, hence $G$ does too.
\end{proof}

Assume that $G$ is a para-$\mathfrak{Nil}$-Hamiltonian group and $H\nno G$. If $|H|$ is finite then $H$ is soluble, for $H$ is nilpotent or minimal-non-nilpotent. So, all finite insoluble subgroups of $G$ are normal.

\begin{theorem}\label{T-3.10}
Let $G$ be a para-$\mathfrak{Nil}$-Hamiltonian group. If $G$ has an insoluble subgroup of finite order, then $G$ is finite and so $G\cong A_5$ or $\SL(2,5)$.
\end{theorem}
\begin{proof}
Assume that $H$ is a finite insoluble. Then $H$ is a  normal subgroup of $G$. As $H$ is para-$\nil$-Hamiltonian, so $H\cong A_5$ or $\SL(2,5)$, by \cite[Theorem 3.6]{DM}. Now we will proceed similarly to the proof of Theorem 3.6 in \cite{DM}, without any changes, and demonstrate that $G$ is finite. Hence, we get $G\cong A_5$ or $\SL(2,5)$.

First, assume that $H\cong A_5$. If $\C_G(H)\neq 1$, then $H\C_G(H)$ is a subgroup of $G$ which contain a non-normal subgroup $K\C_G(H)$, where $K\cong A_4$, which is not  minimal non-nilpotent. So $\C_G(H)=1$ and $G$ is embedded in $\Aut(H)\cong S_5$ is finite by $N/C$-theorem.

Now assume that $H\cong \SL(2,5)$. Obviously  $1\neq Z=Z(G)\sub\C_G(H)$. If $Z\neq \C_G(H)$, then $G/Z$ has a subgroup isomorphic to $A_5\times \C_G(N)/Z$, which leads to a contradiction similar to the previous case. Therefore, $\C_G(H)=Z(H)$ is of order $2$. By using $N/C$-theorem, $G/Z$  is embedded in  $\Aut(H)\cong S_5$. Therefore $G$ is finite again.
\end{proof}
By the above theorem, any finite subgroup of infinite para-$\nil$-Hamiltonian group, is soluble. Therefore any locally finite infinite para-$\nil$-Hamiltonian group is locally soluble.

\begin{theorem}\label{T-3.11}
Let $G$ be a finitely generated para-$\mathfrak{Nil}$-Hamiltonian group. If $G$ contains a nilpotent subgroup of finite index. Then for some $n$, $\gamma_n(G)$ is finite. In additional, if $G$ is insoluble, then $G$ is finite. 
\end{theorem}
\begin{proof}
Let $H$ be a nilpotent subgroup of finite index in $G$. Then $H_G$ is a finitely generated normal subgroup of $G$ with finite torsion subgroup, $T_H$ say, (by~\cite[2.49(ii)]{DKS}). Now, $\overline{N}=H_G/T_H\neq 1$ is a torsion free nilpotent subgroup of $\overline{G}=G/T_H$.

Assume that  $G/T_H$ is non-nilpotent and $p>|G: H|$ is a prime. According to Lemma~\ref{L-3.2}, for some $i$,  $\overline{G}$ contains a series with non-nilpotent terms of dependent subgroups as follows:
\[\cdots\sub\overline{ K_{i+1}}\sub \overline{K_i}.\] 
Since $\overline{G}$ is para-$\nil$-Hamiltonian the subgroups $\overline{K_j}$ are minimal non-nilpotent for all $j\geq i$. Therefore, for any $j>i$, $\overline{K_i}=\overline{K_j}$. As 
\[\overline{K_{j-1}}=\overline{K_j}=\overline{N}^{p^j}\overline{K_{j-1}}\] and $\overline{N}\cap\overline{K_j}=\overline{N}^{p^j}$, so for all $j> i$, $\overline{N}^{p^i}=\overline{N}^{p^j}$. This is a contradiction, since $\overline{N}$ is residually finite. Therefore $G/T_H$ is nilpotent and for some $n$, $\gamma_n(G)\sub T_H$ is finite.

If $G$ is insoluble, $T_H$ is insoluble and by Theorem~\ref{T-3.10}, $G$ is finite.
\end{proof}
As a remark of  the additional part of Theorem~\ref{T-3.11}, we can prove this part without using Theorem~\ref{T-3.10}. Since $\overline{G}/\overline{N}^{p^j}$ is finite insoluble para-$\nil$-Hamiltonian for any $j\geq i$, so for any $j\geq i$, $\overline{G}/\overline{N}^{p^j}\cong A_5$ or $\SL(2,5)$. Therefore, for any $j> i$, $\overline{N}^{p^{i+1}}=\overline{N}^{p^j}$, and this leads to a contradiction because $\overline{N}$ is residually finite.

\begin{theorem}\label{T-3.12}
Let $G$ be a insoluble para-$\mathfrak{Nil}$-Hamiltonian group. Assume that $X\sub G$ is a non-nilpotent subgroup of finite order.
\begin{itemize}
\item[(i)] If $X^ G$ the normal closer of  $X$ is finite, as well as, $|X^G : X|$ is finite, then $G$ is finite. 
\item[(ii)] If $X\nor G$ or $X\nno G$ but $ X\nor X^G$, then $G$ is finite.
\item[(iii)] Assume that  $X^G$ is of infinite order, then $\gamma_3(G)=\gamma_{\infty}(G)$ is finitely generated perfect insoluble group that its Fitting factor is non-abelian simple. In additional, if $\gamma_3(G)$ is minimal non-nilpotent, then $\Fit(\gamma_3(G))=\Phi(\gamma_3(G))$.
\end{itemize}
\end{theorem}
\begin{proof}
(i) As $G/X^G$ is a Dedekind group, then $X^G$ must be insoluble and so, by Theorem~\ref{T-3.6}, $G$ is finite.

(ii) Obviously, if $X\nor G$, $X^G=X$ is finite, so by (i), $G$ is finite. Otherwise, $X$ is maximal in $X^G$, by \cite[Lemma 4.5]{DM}. So,  if $X\nor X^G$,  again $X^G$ is finite. 

(iii) By assumption, neither $X\nor G$ nor $X\nor X^G$, as well as $|X^G : X|$ is not finite. Since  $X$ is maximal in $X^G$, by \cite[Lemma 4.5]{DM}, $X^G$ is finitely generated. Let $Y\nor X^G$ be of finite index. Then $X^G=XY$, thus $|X^G : Y|\leq |X|$ and so $X^G$ has finitely many normal subgroup of finite index. Hence $J_X$
 the finite residual of $X^G$ is of finite index, so we get $X^G=XJ_X$ and $J_X$ is finitely generate too. Since $X^G/J_X\cong X/(X\cap J_X)$ is soluble, so $J_X$ is insoluble. Therefore $G/J_X$ is a Dedekind group and so $\gamma_3(G)\sub J_X$. As $X$ is maximal subgroup of $X^G$, then any proper normal subgroup of $J_X$ must be contain in $X$ and so is soluble. Therefore $J_X$ is perfect. 
 Also  $\gamma_i(G)$, for any $i\geq 3$ is insoluble, so $\gamma_i(G)=J_X$,  thus $\gamma_3(G)=J_X=\gamma_{\infty}(G)$. 
 
For any proper normal subgroup $N$ of $J_X$, if $\C_{J_X}(N)$ is a proper subgroup of $J_X$, then by the N/C-theorem, $J_X$ is finite, which leads to a contradiction. Therefore, $N$ is a subgroup of the center of $J_X$, and this implies that the Fitting subgroup of $J_X$ is equal to the center of $J_X$, and $J_X/\Fit(J_X)$ is an infinite simple group.
 
 Assume that $J_X$ is minimal non-nilpotent and $\Fit(J_X)\neq \Phi(J_X)$. Then, for a maximal subgroup $M$ of $J_X$, we have $J_X=\Fit(J_X)M$, therefore $J_X/\Fit(J_X)$ is nilpotent a contradiction.
 \end{proof}

\subsection*{Funding }This paper is published as part of research supported by the Research Affairs Office of the University of Tabriz.

\end{document}